\documentclass[12pt,reqno]{article}

\usepackage[usenames]{color}
\usepackage{amssymb}
\usepackage{amsmath}
\usepackage{amsthm}
\usepackage{amsfonts}
\usepackage{amscd}
\usepackage{graphicx}

\usepackage[colorlinks=true,
linkcolor=webgreen,
filecolor=webbrown,
citecolor=webgreen]{hyperref}

\definecolor{webgreen}{rgb}{0,.5,0}
\definecolor{webbrown}{rgb}{.6,0,0}

\usepackage{color}
\usepackage{fullpage}
\usepackage{float}

\usepackage{graphics}
\usepackage{latexsym}
\usepackage{epsf}
\usepackage{breakurl}

\newcommand{\seqnum}[1]{\href{https://oeis.org/#1}{\rm \underline{#1}}}

\usepackage{tikz}
\usepackage{tkz-tab}
\usetikzlibrary{matrix}
\usetikzlibrary{calc}
\usetikzlibrary{fit}
\newcommand{\Lar}{0.5}

\begin{document}

\theoremstyle{plain}
\newtheorem{theorem}{Theorem}
\newtheorem{corollary}[theorem]{Corollary}
\newtheorem{lemma}[theorem]{Lemma}
\newtheorem{proposition}[theorem]{Proposition}

\theoremstyle{definition}
\newtheorem{definition}[theorem]{Definition}
\newtheorem{definitions}[theorem]{Definitions}
\newtheorem{example}[theorem]{Examples}
\newtheorem{conjectures}[theorem]{Conjectures}

\theoremstyle{remark}
\newtheorem{remark}[theorem]{Remark}
\newtheorem{question}[theorem]{Question}


\begin{center}
\vskip 1cm{\LARGE\bf 
Two New Integer Sequences Related to Crossroads and Catalan Numbers \\
\vskip .1in
}
\vskip 1cm
\large
Julien Rouyer and Alain Ninet\\
CReSTIC and LMR\\
Universit\'e de Reims Champagne-Ardenne \\
UFR Sciences Exactes et Naturelles\\
Chemin des Rouliers\\
51100 Reims \\
France \\
\href{mailto: julien.rouyer@univ-reims.fr}{\tt julien.rouyer@univ-reims.fr}\\
\href{mailto: alain.ninet@univ-reims.fr}{\tt alain.ninet@univ-reims.fr}\\
\emph{Second version Submitted to the Journal of Integer Sequences on February 12, 2024.}
\end{center}

\vskip .2 in

\begin{abstract}
 The {\it marriageable singles} sequence represents the number of noncrossing partitions of the finite set $\{1,\ldots,n\}$ in which some pair of singleton blocks can be joined while remaining noncrossing.
 The {\it lonely  singles} sequence represents the number of all the other noncrossing partitions of the finite set $\{1,\ldots,n\}$ and is the difference between the Catalan numbers sequence and the marriageable singles sequence. 
 The 14 first terms of these sequences are given, as well as some of their properties.
 These sequences appear when one wants to count the number of ways to cross simultaneously certain road intersections.
\end{abstract}

\section{Introduction}

The number of noncrossing partitions of the finite set $\{1,\ldots,n\}$ (with $n$ any positive integer) is very well known to be the Catalan number $C_n$. See, for example, Stanley \cite[entry~159, p.\ 43]{Stanley2015} and Roman \cite[pp.\ 51--60]{Roman2015} for a quick introduction to Catalan numbers and noncrossing partitions.
Noncrossing partitions have been hugely studied, since at least Becker \cite{Becker1952}, where they are called \emph{planar rhyme schemes} but their systematic study began with Kreweras \cite{Kreweras1972} and Poupard \cite{Poupard1972}. Simion \cite{Simion2000} presents a summary of related results available in 2000 and some further work can be found in McCammond \cite{Mccammond2006}, Callan \cite{Callan2008}  and Kim \cite{Kim2011}.

The study of combinatorial properties of crossroads led us to determine the number of noncrossing partitions such that no pair of singleton blocks $\{i\}$ and $\{j\}$ (with $i\neq j$) can be merged into the pair $\{i,j\}$ while the partition remains noncrossing.
These noncrossing partitions appear when one wants to determine the number of possible manners to cross simultaneously a road intersection in which entries and exits are alternated, with the constraint that U-turns are prohibited. For a quick introduction to road intersection crossing management for intelligent vehicles, see Rouyer et al.\ \cite{Rouyer2022} and Bai et al.\ \cite{Bai2021}. 

In what follows and the next section, we give definitions and quick examples. We then proove some properties of both marriageable singles and lonely singles sequences. At the end, we give the first values of those sequences and formulate several conjectures concerning them.

\begin{definition}
For all $n\in\mathbb{N}$, $[n]$ denotes the $n$-set $\{1,\dots,n\}$. In particular, $[0]=\emptyset$.
\end{definition}

\begin{definitions}
Let $n$ be a non-negative integer and let $\pi=\{A_1,\dots,A_k\}$ be a partition of $[n]$ (i.e., $\cup_{i=1}^k A_i=[n]$ and $\forall\, 1\leq i< j\leq k, A_i\cap A_j=\emptyset$ and $\forall\, 1\leq i\leq k, A_i\neq\emptyset$), this partition $\pi$ is said to be a
\begin{itemize}
\item crossing partition if
$$\exists\, 1\leq i\leq k, \exists\, 1\leq j\leq k,  i\neq j, \exists\, a<b\in A_i, \exists\, c<d\in A_j, a<c<b<d,$$
\item noncrossing partition  if
\begin{align*}
\forall\, 1\leq i<j\leq k, \forall\, a<b\in A_i, \forall\, c<d\in A_j,\ & \phantom{\text{ or }} a<b<c<d,\\
 & \text{ or } c<d<a<b,\\
 & \text{ or } a<c<d<b,\\
 & \text{ or } c<a<b<d.
\end{align*}
\end{itemize}
\end{definitions}

\begin{definitions} 
A noncrossing partition $\pi$ of $[n]$ is called \emph{marriageable singles partition} if there exists a pair of singleton blocks $\{i\}$ and $\{j\}$ in $\pi$ that can be joined while remaining noncrossing.

More precisely, let $\pi=\{\,A_1,\dots,A_k\,\}$ be a noncrossing partition of $[n]$ with at least two singleton blocks $A_1=\{i\}$ and $A_2=\{j\}$. Then $\pi$ is a marriageable singles partition iif $\pi'=\{\,\{i,j\},A_3,\dots,A_k\,\}$ is a noncrossing partition of $[n]$.

Conversely, a noncrossing partition of $[n]$ is called \emph{lonely singles partition} if it is not a marriageable singles partition.

Let $M_n$ and $L_n$ be the number of marriageable singles and lonely singles partitions of $[n]$ respectively. The \emph{marriageable singles} and \emph{lonely singles sequences} are $(M_n)_{n \geq 0}$ and $(L_n)_{n \geq 0}$ respectively.
\label{def:LnMn}
\end{definitions}

\begin{remark}
A partition that contains at most one singleton block is clearly a lonely singles partition. In an equivalent way, a marriageable singles partition contains at least two singleton blocks. 
\end{remark}

\begin{example}
The set $[4]$ has $5$ marriageable singles partitions which are
\begin{align*}
\{\,\{1\}, \{2\}, \{3\}, \{4\}\,\} &, &\{\,\{1\}, \{2\}, \{3,4\}\,\} &, &\{\,\{1\}, \{2,3\}, \{4\}\,\}, \\
\{\,\{1,4\}, \{2\}, \{3\}\,\} &, &\{\,\{1,2\}, \{3\}, \{4\}\,\}&. &
\end{align*}

For example, $\pi=\{\,\{1,2\}, \{3\}, \{4\}\,\}$ is a marriageable singles partition because the singleton blocks $\{3\}$ and $\{4\}$ can be joined to give the noncrossing partition $\pi'=\{\,\{1,2\}, \{3,4\}\,\}$.

The set $[4]$ has $9$ lonely singles partitions which are
\begin{align*}
\{\,\{1, 2, 3, 4\}\,\}&, & \{\,\{1, 2, 3\}, \{4\}\,\}&, & \{\,\{1, 2, 4\}, \{3\}\,\}, \\ 
\{\,\{1, 3, 4\}, \{2\}\,\}&, & \{\,\{1\}, \{2, 3, 4\}\,\}&, & \{\,\{1, 2\}, \{3, 4\}\,\},\\
\{\,\{1, 4\}, \{2, 3\}\,\}&, & \{\,\{1, 3\}, \{2\}, \{4\}\,\}&, & \{\,\{1\}, \{2, 4\}, \{3\}\,\}. 
\end{align*}
The first seven have no pair of singleton blocks and are clearly not marriageable singles partitions. The last two have only one pair of singleton blocks, but after merging it, they both give $\{\,\{1,3\}, \{2, 4\}\,\}$ which is a crossing partition.
\end{example}

\begin{lemma}
Let $C_n$ (for $n\geq 0$) denote the number of noncrossing partitions of $[n]$. Then, we have $C_n=L_n+M_n$.
\label{lemma_somme}
\end{lemma}

\begin{proof}
As the set of noncrossing partitions is the disjoint union of the sets of lonely singles and marriageable singles partitions, the result follows immediately.
\end{proof}

\begin{remark}
It is well known that $C_n=\frac{1}{n+1}\binom{2n}{n}$ is the $n$th Catalan number. The sequence $(C_n)$ is referenced as \seqnum{A000108} in the On-Line Encyclopedia of Integer Sequences \cite{oeis}.
\end{remark}

\begin{remark}
The unique partition of $[0]=\emptyset$ is the empty partition $\emptyset$. This partition is noncrossing (the first Catalan number is $C_0=1$) and it is a lonely singles partition.
\end{remark}

\section{Standard road intersection}

Noncrossing complete matchings of $2n$ points lying on a line and noncrossing complete set of chords are two of the many standard combinatorial objects counted by the Catalan numbers. Noncrossing complete matchings and noncrossing complete set of chords are in natural bijection with the noncrossing partitions, as explained by Stanley \cite[entries 59 and 61, p.\ 28]{Stanley2015}. 

We introduce here some notions on road intersections corresponding to noncrossing sets of chords.

\begin{definition}
A road intersection with $n$ entries and $n$ exits alternated is called a {\it standard road intersection of size} $n$. 

Let $E_1,\dots,E_n$ denote the entries and $X_1,\dots,X_n$ denote the exits of a standard road intersection of size $n$. These entries and exits are numbered clockwise ({\bf Figure} \ref{fig:inter} gives such a representation of a standard road intersection of size $n=4$).

We represent graphically a way to cross simultaneously a standard road intersection by a bipartite graph (see {\bf Figures} \ref{fig:cercle} and \ref{fig:cerclenonmax} for two examples with $n=4$). In these graphs, black vertices represent entries and white vertices represent exits. Each edge represents the crossing of the intersection by a vehicle going from an entry to an exit.

Each entry can be connected to each exit, unless restrictions are indicated.
\end{definition}

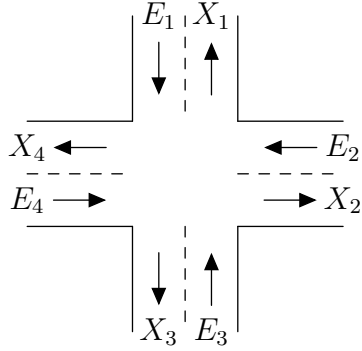
\begin{figure}
\centering
\begin{tikzpicture}[line cap=round,line join=round,>=triangle 45,x=0.35cm,y=0.35cm]
\clip(2,-3) rectangle (16,11);
\draw [line width=\Lar pt] (3,6)-- (7,6);
\draw [line width=\Lar pt] (3,2)-- (7,2);
\draw [line width=\Lar pt] (7,6)-- (7,10);
\draw [line width=\Lar pt] (11,6)-- (11,10);
\draw [line width=\Lar pt] (7,2)-- (7,-2);
\draw [line width=\Lar pt] (11,2)-- (11,-2);
\draw [line width=\Lar pt] (11,6)-- (15,6);
\draw [line width=\Lar pt] (11,2)-- (15,2);
\draw [line width=\Lar pt,dash pattern=on 4pt off 4pt] (3,4)-- (7,4);
\draw [line width=\Lar pt,dash pattern=on 4pt off 4pt] (11,4)-- (15,4);
\draw [line width=\Lar pt,dash pattern=on 4pt off 4pt] (9,10)-- (9,6);
\draw [line width=\Lar pt,dash pattern=on 4pt off 4pt] (9,2)-- (9,-2);
\draw [->,line width=\Lar pt] (4,3) -- (6,3);
\draw [->,line width=\Lar pt] (12,3) -- (14,3);
\draw [->,line width=\Lar pt] (14,5) -- (12,5);
\draw [->,line width=\Lar pt] (6,5) -- (4,5);
\draw [->,line width=\Lar pt] (10,7) -- (10,9);
\draw [->,line width=\Lar pt] (8,9) -- (8,7);
\draw [->,line width=\Lar pt] (8,1) -- (8,-1);
\draw [->,line width=\Lar pt] (10,-1) -- (10,1);
\draw[color=black] (8,10) node {$E_1$};
\draw[color=black] (10,10) node {$X_1$};
\draw[color=black] (15,5) node {$E_2$};
\draw[color=black] (15,3) node {$X_2$};
\draw[color=black] (10,-2) node {$E_3$};
\draw[color=black] (8,-2) node {$X_3$};
\draw[color=black] (3,3) node {$E_4$};
\draw[color=black] (3,5) node {$X_4$};
\end{tikzpicture}
\caption{A Standard Road Intersection of size $n=4$.}
\label{fig:inter}
\end{figure}

\begin{definition}
For a given standard road intersection, any edge starting from one entry $E_i$ and ending to one exit $X_j$ is a \emph{lane}. $E_iX_j$ denotes such a lane.
\end{definition}

\begin{definition}
For a given standard road intersection, any lane of the type $E_iX_i$, i.e., any edge starting from one entry $E_i$ and ending to exit $X_i$ is a \emph{U-turn}.
\end{definition}

\begin{definition}
A standard road intersection of size $n$ where U-turns are forbidden is a \emph{restricted standard road intersection}.
\end{definition}

\begin{definition}
For a given standard road intersection, a \emph{Maximal Set of Lanes}, abbreviated \emph{MSL}, is a set of noncrossing lanes (i.e., each pair of lanes have no common point) such that any additional lane would cross at least one of them (i.e., would have a common point with at least one of them).
\end{definition}

\begin{remark}
An MSL corresponds to a set of $n$ nonintersecting chords (or a noncrossing complete matching on $2n$ vertices), see Stanley \cite[entries 59 and 61, p.\ 28]{Stanley2015}. Three examples of MSL are given by {\bf Figures} \ref{fig:cercle}, \ref{fig:cerclenonmax} and \ref{fig:cerclenonmaxmod}.
\end{remark}

\begin{definition}
For a given standard road intersection, an \emph{MSL} is said to be \emph{absolute} when it does not contain two U-turns $E_iX_i$ and $E_jX_j$ (with $i\neq j$) that can be changed into two lanes $E_iX_j$ and $E_jX_i$ to give another MSL.
\label{def:abs}
\end{definition}

\begin{figure}
    \begin{minipage}[c]{.46\linewidth}
\centering
\begin{tikzpicture}[line cap=round,line join=round,>=triangle 45,x=0.5cm,y=0.5cm]
\clip(-4,-4) rectangle (4,4);
\draw [line width=10pt] (0,0) circle (6cm);
\draw (-1,3)--(-1,-2.83);
\draw (1,-3)--(2.87,-1.13);
\draw (3,1)--(1.13,2.87);
\draw (-3,-1)--(-3,0.83);
\begin{scriptsize}
\fill [black] (-1,3) circle (2.5pt);
\draw[black] (-1,3.7) node {$E_1$};
\draw [black] (1,3) circle (2.5pt);
\draw[black] (1,3.7) node {$X_1$};
\fill [black] (3,1) circle (2.5pt);
\draw[black] (3.7,1) node {$E_2$};
\draw [black] (3,-1) circle (2.5pt);
\draw[black] (3.7,-1) node {$X_2$};
\fill [black] (1,-3) circle (2.5pt);
\draw[black] (1,-3.7) node {$E_3$};
\draw [black] (-1,-3) circle (2.5pt);
\draw[black] (-1,-3.7) node {$X_3$};
\fill [black] (-3,-1) circle (2.5pt);
\draw[black] (-3.7,-1) node {$E_4$};
\draw [black] (-3,1) circle (2.5pt);
\draw[black] (-3.7,1) node {$X_4$};
\end{scriptsize}
\end{tikzpicture}
\caption{A bipartite graph associated with the intersection represented in Figure \ref{fig:inter}, with an example of absolute MSL corresponding to the noncrossing partition $\{\,\{1,2,3\},\{4\}\,\}$.}
\label{fig:cercle}
    \end{minipage}\hfill
    \begin{minipage}[c]{.46\linewidth}
\centering
\begin{tikzpicture}[line cap=round,line join=round,>=triangle 45,x=0.5cm,y=0.5cm]
\clip(-4,-4) rectangle (4,4);
\draw [line width=10pt] (0,0) circle (6cm);
\draw (-1,3)--(2.85,-0.85);
\draw (3,1)--(1.13,2.87);
\draw (-0.83,-3)--(1,-3);
\draw (-3,-1)--(-3,0.83);
\begin{scriptsize}
\fill [color=black] (-1,3) circle (2.5pt);
\draw[color=black] (-1,3.7) node {$E_1$};
\draw [color=black] (1,3) circle (2.5pt);
\draw[color=black] (1,3.7) node {$X_1$};
\fill [color=black] (3,1) circle (2.5pt);
\draw[color=black] (3.7,1) node {$E_2$};
\draw [color=black] (3,-1) circle (2.5pt);
\draw[color=black] (3.7,-1) node {$X_2$};
\fill [color=black] (1,-3) circle (2.5pt);
\draw[color=black] (1,-3.7) node {$E_3$};
\draw [color=black] (-1,-3) circle (2.5pt);
\draw[color=black] (-1,-3.7) node {$X_3$};
\fill [color=black] (-3,-1) circle (2.5pt);
\draw[color=black] (-3.7,-1) node {$E_4$};
\draw [color=black] (-3,1) circle (2.5pt);
\draw[color=black] (-3.7,1) node {$X_4$};
\end{scriptsize}
\end{tikzpicture}
\caption{A bipartite graph associated with the intersection represented in Figure \ref{fig:inter}, with an example of a nonabsolute MSL corresponding to the noncrossing partition $\{\,\{1,2\},\{3\},\{4\}\,\}$.}
\label{fig:cerclenonmax}
    \end{minipage}\hfill
\end{figure}

\begin{figure}
\begin{minipage}[c]{.46\linewidth}
\centering
\begin{tikzpicture}[line cap=round,line join=round,>=triangle 45,x=0.5cm,y=0.5cm]
\clip(-4,-4) rectangle (4,4);
\draw [line width=10pt] (0,0) circle (6cm);
\draw (0,3)--(0,-3);
\draw (0,-3)--(3,0);
\draw (3,0)--(0,3);
\draw (-3,0)--(-3,0);
\begin{scriptsize}
\fill [color=black] (0,3) circle (2.5pt);
\draw[color=black] (0,3.7) node {$1$};
\fill [color=black] (3,0) circle (2.5pt);
\draw[color=black] (3.7,0) node {$2$};
\fill [color=black] (0,-3) circle (2.5pt);
\draw[color=black] (0,-3.7) node {$3$};
\fill [color=black] (-3,0) circle (2.5pt);
\draw[color=black] (-3.7,0) node {$4$};
\end{scriptsize}
\end{tikzpicture}
\caption{Simplified graph of Figure \ref{fig:cercle}, showing explicitly the noncrossing lonely singles partition $\{\,\{1,2,3\},\{4\}\,\}$.}
\label{fig:cerclesimp}
    \end{minipage}\hfill
    \begin{minipage}[c]{.46\linewidth}
\centering
\begin{tikzpicture}[line cap=round,line join=round,>=triangle 45,x=0.5cm,y=0.5cm]
\clip(-4,-4) rectangle (4,4);
\draw [line width=10pt] (0,0) circle (6cm);
\draw (3,0)--(0,3);
\begin{scriptsize}
\fill [color=black] (0,3) circle (2.5pt);
\draw[color=black] (0,3.7) node {$1$};
\fill [color=black] (3,0) circle (2.5pt);
\draw[color=black] (3.7,0) node {$2$};
\fill [color=black] (0,-3) circle (2.5pt);
\draw[color=black] (0,-3.7) node {$3$};
\fill [color=black] (-3,0) circle (2.5pt);
\draw[color=black] (-3.7,0) node {$4$};
\end{scriptsize}
\end{tikzpicture}
\caption{Simplified graph of Figure \ref{fig:cerclenonmax}, showing explicitly the noncrossing marriageable singles partition $\{\,\{1,2\},\{3\},\{4\}\,\}$.}
\label{fig:cerclenonmaxsimp}
    \end{minipage}
\end{figure}

\begin{figure}
    \begin{minipage}[c]{.46\linewidth}
\centering
\begin{tikzpicture}[line cap=round,line join=round,>=triangle 45,x=0.5cm,y=0.5cm]
\clip(-4,-4) rectangle (4,4);
\draw [line width=10pt] (0,0) circle (6cm);
\draw (-1,3)--(2.85,-0.85);
\draw (3,1)--(1.13,2.87);
\draw (-3,-1)--(-1.13,-2.87);
\draw (-2.87,0.83)--(1,-3);
\begin{scriptsize}
\fill [color=black] (-1,3) circle (2.5pt);
\draw[color=black] (-1,3.7) node {$E_1$};
\draw [color=black] (1,3) circle (2.5pt);
\draw[color=black] (1,3.7) node {$X_1$};
\fill [color=black] (3,1) circle (2.5pt);
\draw[color=black] (3.7,1) node {$E_2$};
\draw [color=black] (3,-1) circle (2.5pt);
\draw[color=black] (3.7,-1) node {$X_2$};
\fill [color=black] (1,-3) circle (2.5pt);
\draw[color=black] (1,-3.7) node {$E_3$};
\draw [color=black] (-1,-3) circle (2.5pt);
\draw[color=black] (-1,-3.7) node {$X_3$};
\fill [color=black] (-3,-1) circle (2.5pt);
\draw[color=black] (-3.7,-1) node {$E_4$};
\draw [color=black] (-3,1) circle (2.5pt);
\draw[color=black] (-3.7,1) node {$X_4$};
\end{scriptsize}
\end{tikzpicture}
\caption{Modification of Figure \ref{fig:cerclenonmax} to obtain an absolute MSL corresponding to the noncrossing partition $\{\,\{1,2\},\{3,4\}\,\}$.}
\label{fig:cerclenonmaxmod}
    \end{minipage}\hfill
    \begin{minipage}[c]{.46\linewidth}
\centering
\begin{tikzpicture}[line cap=round,line join=round,>=triangle 45,x=0.5cm,y=0.5cm]
\clip(-4,-4) rectangle (4,4);
\draw [line width=10pt] (0,0) circle (6cm);
\draw (3,0)--(0,3);
\draw (-3,0)--(0,-3);
\begin{scriptsize}
\fill [color=black] (0,3) circle (2.5pt);
\draw[color=black] (0,3.7) node {$1$};
\fill [color=black] (3,0) circle (2.5pt);
\draw[color=black] (3.7,0) node {$2$};
\fill [color=black] (0,-3) circle (2.5pt);
\draw[color=black] (0,-3.7) node {$3$};
\fill [color=black] (-3,0) circle (2.5pt);
\draw[color=black] (-3.7,0) node {$4$};
\end{scriptsize}
\end{tikzpicture}
\caption{Simplified graph of Figure \ref{fig:cerclenonmaxmod}, showing explicitly the noncrossing lonely singles partition $\{\,\{1,2\},\{3,4\}\,\}$.}
\label{fig:cerclenonmaxmodsimp}
    \end{minipage}
\end{figure}

\begin{lemma}\label{lem:catalan}
For a given standard road intersection of size $n$, the set of MSL is in one-to-one correspondence with the set of noncrossing partitions of $[n]$. The number of MSL of a standard road intersection of size $n$ is equal to the Catalan number $C_n$.
\end{lemma}

\begin{proof}
As a MSL can be seen as a set of $n$ nonintersecting chords joining $2n$ points, this result is very well known. See Stanley \cite[entry~59, p.\ 28]{Stanley2015}.
\end{proof}

The following three corollaries are just reinterpretations of the lonely singles and marriageable singles definitions in the language of road intersections.

\begin{corollary}
$M_n$ is the number of nonabsolute MSL of a standard road intersection of size~$n$.
\end{corollary}

\begin{corollary}
$L_n$ is the number of absolute MSL of a standard road intersection of size~$n$.
\end{corollary}

\begin{corollary}
$L_n$ is the number of MSL for a restricted standard road intersection of size~$n$.
\end{corollary}

\section{Properties}

\begin{proposition}
Both sequences $(L_n)$ and $(M_n)$ are increasing: $L_n<L_{n+1}$ for all $n\geq 2$ and $M_n<M_{n+1}$ for all $n\geq 3$.
\label{prop:increasing}
\end{proposition}

\begin{proof}
We build a simple injective map $f_n$ from the set $LS_n$ of the lonely singles partitions of $[n]$ to the set $LS_{n+1}$ of the lonely singles partitions of $[n+1]$ by merging the singleton $\{n+1\}$ to the unique element $A_1$ of a partition $\pi$ that contains the number $1$ (one could equally prefer to use the number $n$ instead of the number 1: the idea is to stick the number $n+1$ to one of its two direct neighbours 1 or $n$):
\begin{align*}
LS_n & \to LS_{n+1}\\
f_n\colon \pi=\{\,A_1,\dots,A_k\,\} & \mapsto \{\,A_1\cup\{n+1\},\dots,A_k\,\}
\end{align*}
where $1\in A_1$ and $A_1\cup\dots\cup A_k=[n]$ and $A_i\cap A_j=\emptyset$ for all $i\neq j$. For example,
\begin{align*}
f_3\left(\{\,\{1,2\},\{3\}\,\}\right) &= \{\,\{1,2,4\},\{3\}\,\}\\
f_4\left(\{\,\{1\},\{2,4\},\{3\}\,\}\right)&= \{\,\{1,5\},\{2,4\},\{3\}\,\}\\
\end{align*}
The map $f_n$ is clearly injective and any lonely singles partition $\pi$ gives a lonely singles partition $f_n(\pi)$. As $L_n$ is the cardinality of $LS_n$ and $L_{n+1}$ is the cardinality of $LS_{n+1}$, we obtain $L_n\leq L_{n+1}$. 

In a similar way, we build a simple injective map $g_n$ from the set $MS_n$ of the marriageable singles partitions of $[n]$ to the set $MS_{n+1}$ of the marriageable singles partitions of $[n+1]$ by adding the singleton $\{n+1\}$ to any partition $\pi$:
\begin{align*}
MS_n & \to MS_{n+1}\\
g_n \colon \pi & \mapsto \pi\cup \{\,\{n+1\}\,\}
\end{align*}
For example
\begin{align*}
g_4\left(\{\,\{1\},\{2,3\},\{4\}\,\}\right)&= \{\,\{1\},\{2,3\},\{4\},\{5\}\,\}
\end{align*}
The map $g_n$ is clearly injective and any pair of marriageable singletons $\{i\}$ and $\{j\}$ of $\pi\in MS_n$ remains marriageable as elements of $g_n(\pi)\in MS_{n+1}$. As $M_n$ is the cardinality of $MS_n$ and $M_{n+1}$ is the cardinality of $MS_{n+1}$, we obtain $M_n\leq M_{n+1}$.

More precisely, we have $L_n<L_{n+1}$ for all $n\geq 2$ and $M_n<M_{n+1}$ for all $n\geq 3$: it is easy to build a lonely singles and a marriageable singles partitions of $[n+1]$ that are not in the images of the maps $f_n$ and $g_n$, e.g.\ respectively $\{\,[n],\{n+1\}\,\}$ and $\{\,\{1\},\dots,\{n-1\},\{n,n+1\}\,\}$.
\end{proof}

\begin{corollary}
$$\lim_{n\to+\infty}M_n=\lim_{n\to+\infty}L_n=+\infty$$
\end{corollary}
\begin{proof}
It is an an immediate consequence of {\bf Proposition \ref{prop:increasing}}.
\end{proof}

\begin{proposition}
For all $n\geq 0$, we have $C_n+3M_n\leq M_{n+2}$.
\end{proposition}

\begin{proof}
We build four simple injective maps $h_n, i_n, j_n$ and $k_n$ with disjoint images included in the set $MS_{n+2}$ of the marriageable singles partitions of $[n+2]$.

\begin{itemize}
\item The map $h_n$ is defined on the set $NC_n$ of all noncrossing partitions of $[n]$ by adding both the singletons $\{n+1\}$ and $\{n+2\}$ to any noncrossing partition $\pi$:
\begin{align*}
NC_n&\to MS_{n+2}\\
h_n\colon \pi&\mapsto \pi\cup \{\,\{n+1\}, \{n+2\}\,\}
\end{align*}
The maps $i_n, j_n$ and $k_n$ are defined on the set $MS_n$ of the marriageable singles partitions of $[n]$.

\item The map $i_n$ adds the pair $\{n+1, n+2\}$ to a marriageable singles partition $\pi$:
\begin{align*}
MS_n &\to MS_{n+2}\\
i_n\colon \pi &\mapsto \pi\cup \{\,\{n+1,n+2\}\,\}
\end{align*}

\item The map $j_n$ merges the singleton $\{n+2\}$ with the unique element $A_1$ of a marriageable singles partition $\pi$ that contains the number $1$ and adds the singleton $\{n+1\}$ to $\pi$:
\begin{align*}
MS_n&\to MS_{n+2}\\
j_n\colon \pi=\{\,A_1,\dots,A_k\,\}&\mapsto \{\,A_1\cup\{n+2\},A_2,\dots,A_k, \{n+1\}\,\}
\end{align*}

where $1\in A_1$, $A_1\cup\dots\cup A_k=[n]$ and $A_i\cap A_j=\emptyset$ for all $i\neq j$.

\item the map $k_n$ merges the singleton $\{n+1\}$ with the unique element $A_1$ of a marriageable singles partition $\pi$ that contains the number $n$ and adds the singleton $\{n+2\}$ to $\pi$:
\begin{align*}
MS_n &\to MS_{n+2}\\
k_n : \pi=\{\,A_1,\dots,A_k\,\} &\mapsto \{\,A_1\cup\{n+1\},A_2,\dots,A_k, \{n+2\}\,\}
\end{align*}

where $n\in A_1$, $A_1\cup\dots\cup A_k=[n]$ and $A_i\cap A_j=\emptyset$ for all $i\neq j$.
\end{itemize}
For example, 
\begin{align*}
h_4\left(\{\,\{1,2,3\},\{4\}\,\}\right) &= \{\,\{1,2,3\},\{4\},\{5\},\{6\}\,\}\\
i_4\left(\{\,\{1,2\},\{3\},\{4\}\,\}\right) &= \{\,\{1,2\},\{3\},\{4\},\{5,6\}\,\}\\
j_4\left(\{\,\{1,2\},\{3\},\{4\}\,\}\right) &= \{\,\{1,2,6\},\{3\},\{4\},\{5\}\,\}\\
k_4\left(\{\,\{1,2\},\{3\},\{4\}\,\}\right) &= \{\,\{1,2\},\{3\},\{4,5\},\{6\}\,\}
\end{align*}

We have the immediate following properties:
\begin{itemize}
\item $h_n(\pi)$ is a marriageable singles partition of $[n+2]$ for all noncrossing partition $\pi$ of $[n]$,
\item $i_n(\pi), j_n(\pi)$ and $k_n(\pi)$ are marriageable singles partitions of $[n+2]$ for all marriageable singles partition $\pi$ of $[n]$ (the marriageable pairs can change depending on the map selected),
\item $h_n, i_n, j_n$ and $k_n$ are injective maps; thus,
\begin{itemize}
\item $NC_n$ and $h_n(NC_n)$ have the same cardinality $C_n$,
\item $MS_n$ and $i_n(MS_n)$ and $j_n(MS_n)$ and $k_n(MS_n)$ have the same cardinality $M_n$,
\end{itemize}
\item the sets $h_n(NC_n)$, $i_n(MS_n)$, $j_n(MS_n)$ and $k_n(MS_n)$ are disjoint; thus,
$$ h_n(NC_n)\sqcup i_n(MS_n)\sqcup j_n(MS_n)\sqcup k_n(MS_n)\subset MS_{n+2}, $$
\item as $M_{n+2}$ is the cardinality of $MS_{n+2}$, we obtain $C_n+3M_n\leq M_{n+2}$.
\end{itemize}
\end{proof}

\begin{definition}
For all nonnegative integers $n$, $m$ and $k$, let $NC_{n,m,k}$ be the number of noncrossing partitions of $[n]$ in $m$ classes with $k$ singleton blocks.
\end{definition}

\begin{proposition}
For all nonnegative integers $n$ and $m$ with $(n,m)\neq(0,0)$, the number of noncrossing partitions of $[n]$ in $m$ classes with no singleton block is 
$$NC_{n,m,0}=\frac{1}{n-m+1}\binom{n}{m}\binom{n-m-1}{m-1}.$$
When $(n,m)\neq(1,1)$, the number of noncrossing partitions of $[n]$ in $m$ classes with exactly one singleton block is $$NC_{n,m,1}=\binom{n}{m-1}\binom{n-m-1}{m-2}.$$
\label{prop:atmost1}
\end{proposition}

\begin{proof}
Poupard \cite{Poupard1972} proved that, when $n\geq 1$ and $m\geq 1$, the number of noncrossing partitions of $[n]$ in $m$ classes with no singleton block is $NC_{n,m,0}=\frac{1}{n-m+1}\binom{n}{m}\binom{n-m-1}{m-1}$, and $n\geq 2m$. When $n=0$ and $m\geq 1$ or $n\geq 1$ and $m=0$ or $m\geq \lfloor \frac{n}{2}\rfloor+1$, this equality stands and gives a number of such partitions equal to $0$.

When $n\geq 1$ and $m\geq 1$, the set of noncrossing partitions of $[n]$ in $m$ classes with exactly one singleton block is clearly in one-to-one correspondence with the set of couples $(\pi,i)$, where $\pi$ is any noncrossing partition of $[n-1]$ in $m-1$ classes with no singleton block and $1\leq i\leq n$.
Then, when $(n,m)\neq (1,1)$, we obtain that the number of noncrossing partitions of $[n]$ in $m$ classes with exactly one singleton block is $NC_{n,m,1}=n\times NC_{n-1,m-1,0}=\frac{n}{n-m+1}\binom{n-1}{m-1}\binom{n-m-1}{m-2}=\binom{n}{m-1}\binom{n-m-1}{m-2}$.
\end{proof}

\begin{remark}
Here are a few details on specific cases:
\begin{itemize}
\item The empty partition is the unique noncrossing partition of the empty set $[0]$ in $0$ class with no singleton block. Thus, $NC_{0,0,0}=1$.
\item The set $[1]$ has a unique noncrossing partition. This partition has $m=1$ class and it is a singleton block. Thus, $NC_{1,1,1}=1$ and $NC_{1,m,k}=0$ when $(m,k)\neq(1,1)$ and, in particular, $NC_{1,1,0}=0$.
\item For all positive integer $n$, $[n]$ has no noncrossing partition in $m=0$ class. Thus, for all nonnegative integer $k$, $NC_{n,0,k}=0$ and in particular $NC_{n,0,0}=0$.
\end{itemize}
\end{remark}

\begin{proposition}
For all $n\geq 2$, 
$$L_n\geq \sum_{m=1}^{\lfloor\frac{n}{2}\rfloor}\frac{1}{n-m+1}\binom{n}{m}\binom{n-m-1}{m-1}
+ \sum_{m=2}^{\lfloor\frac{n+1}{2}\rfloor}\binom{n}{m-1}\binom{n-m-1}{m-2}.$$
\end{proposition}
\begin{proof}
As any noncrossing partition with at most one singleton block is a lonely singles partition, we have $L_n\geq \sum_{m=1}^{\lfloor \frac{n}{2}\rfloor} NC_{n,m,0} + \sum_{m=1}^{\lfloor \frac{n+1}{2}\rfloor} NC_{n,m,1}$ and the result follows, using {\bf Proposition \ref{prop:atmost1}}.
\end{proof}

\begin{proposition}
For all $n\geq 3$,
\begin{align*}
M_n\geq \sum_{i=1}^{n-1}\sum_{j=i+1}^{n} &\left(\sum_{m=0}^{\lfloor \frac{n+i-j-1}{2}\rfloor} \frac{1}{n+i-j-m}\binom{n+i-j-1}{m}\binom{n+i-j-m-2}{m-1}\right.\\
&\times \left.\sum_{m=0}^{\lfloor \frac{j-i-1}{2}\rfloor}\frac{1}{j-i-m}\binom{j-i-1}{m}\binom{j-i-m-2}{m-1}\right)
\end{align*}
\end{proposition}

\begin{proof}
Let $\pi$ be a marriageable singles partition of $[n]$ with exactly two singleton blocks $\{i\}$ and $\{j\}$ (with $i<j$). The set of such partitions $\pi$ is clearly in one-to-one correspondence with the set of triplets $(\{i,j\},\pi_1,\pi_2)$ where $1\leq i<j\leq n$ and $\pi_1$ and $\pi_2$ are noncrossing partitions of the sets $[n+i-j-1]$ and $[j-i-1]$ respectively, with no singleton blocks. Thus,

$$M_n\geq \sum_{i=1}^{n-1}\sum_{j=i+1}^{n}\left(\sum_{m=0}^{\lfloor \frac{n+i-j-1}{2} \rfloor}NC_{n+i-j-1,m,0}\sum_{m=0}^{\lfloor \frac{j-i-1}{2} \rfloor}NC_{j-i-1,m,0}\right).$$
\end{proof}

\begin{remark}
The previous inequality can be improved, considering three or more singleton blocks that can all be joined.
\end{remark}
\section{Known values and conjectures}

Very few values of the sequences $L_n$ and $M_n$ are known for the moment. They are given in {\bf Table \ref{table}}. $L_n$ and $M_n$ are respectively sequences \seqnum{A363448} and \seqnum{A363449} in the Online Encyclopedia of Integer Sequences (OEIS) \cite{oeis}.

\begin{table}[H]
\centering
\begin{tabular}{|c|c|c|c|c|c|c|c|c|}
\hline
$n$ & $L_n$ & $L_n/L_{n-1}$ & $M_n$ & $M_n/M_{n-1}$ & $C_n$ & $M_n/L_n$ & $M_n/C_n$ & Computing \\
&&&&&&&& time for $M_n$\\
\hline
0&1&  &0&  &1&0& 0 &0s\\
\hline
1&1& 1 &0&  &1&0& 0 &0s\\
\hline
2&1& 1 &1&  &2&1& 0.5 &0s\\
\hline
3&4& 4 &1& 1 &5&0.25& 0.2 &0s\\
\hline
4&9& 2.25 &5& 5 &14&0.56& 0.36 &263$\mu$s\\
\hline
5&26& 2.89 &16& 3.2 &42&0.62& 0.38 &657$\mu$s\\
\hline
6&77& 2.96 &55& 3.44 &132&0.71& 0.42 &1.8ms\\
\hline
7&232& 3.01 &197& 3.58 &429&0.85& 0.46 &7.5ms\\
\hline
8&725& 3.13 &705& 3.58 &1430&0.97& 0.49 &24ms\\
\hline
9&2299& 3.17 &2563& 3.64 &4862&1.11& 0.53 &90ms\\
\hline
10&7401& 3.22 &9395& 3.67 &16796&1.27& 0.56 &507ms\\
\hline
11&22118& 2.99 &36668& 3.90 &58786&1.66& 0.62 &3.7s\\
\hline
12&72766& 3.29 &135246& 3.69 &208012&1.86& 0.65 &39.4s\\
\hline
13& 235124& 3.23 &507776& 3.75 &742900&2.16& 0.68 &15m, 36s\\
\hline
14& 763783& 3.25 &1910657& 3.76 &2674440&2.50& 0.71 &6h, 2m, 40s\\
\hline
\end{tabular}
\caption{Values of $L_n$, $M_n$, $C_n$, $L_n/L_{n-1}$, $M_n/M_{n-1}$, $M_n/L_n$ and $M_n/C_n$ (given with 2 digits), with computing time of $M_n$, for all $n\leq 14$.}
\label{table}
\end{table}

The values of $M_n$ have been obtained with the algorithms available in the attached files, provided in pdf and ipynb formats. The complexity of these algorithms is exponential: they need to be improved and to be run with a more powerful computer to obtain values of $L_n$ and $M_n$ for $n\geq 15$. Calculations were performed under Python 3.7.3 using Jupyter Notebook 5.7.6 with a 3.19 GHz i7-8700 processor and 32 GB RAM. 

\begin{conjectures}
We conjecture the five following propositions:
\begin{align}
&\forall n\in\mathbb{N}, n\geq 9 \Rightarrow M_n > L_n,\label{conj1}\\
&\lim_{n\to+\infty}\frac{M_n}{L_n}=+\infty,\label{conj2}\\
&\lim_{n\to+\infty}\frac{M_n}{C_n}=1,\label{conj3}\\
&\lim_{n\to+\infty}\frac{L_n}{C_n}=0,\label{conj4}\\
&\lim_{n\to+\infty}\frac{M_{n+1}}{M_n}=\lim_{n\to+\infty}\frac{L_{n+1}}{L_n}=4.\label{conj5}
\end{align}
\end{conjectures}

\begin{remark}
Obviously, {\bf conjecture} (\ref{conj2}) implies {\bf conjecture} (\ref{conj1}).
{\bf Table \ref{table}} shows that $\frac{M_n}{L_n}$ grows slowly and {\bf conjecture} \ref{conj2} may stand. More clearly, {\bf Table \ref{table}} shows that $\frac{M_n}{C_n}$ seems to grow quite quickly and let think that {\bf conjecture} \ref{conj3} stands. As $C_n=M_n+L_n$, {\bf conjectures} (\ref{conj2}), (\ref{conj3}) and (\ref{conj4}) are clearly equivalent. It is well known and easy to prove that $\lim_{n\to+\infty}\frac{C_{n+1}}{C_n}=4$. {\bf Table \ref{table}} shows that the similar limits given in {\bf conjecture} (\ref{conj5}) look very realistic as well.
\end{remark}

\section{Acknowledgements}
We would like to thank Fr\'ed\'eric Blanchard and Rupert Wei Tze Yu for their constant help. Our deepest gratitude goes to the anonymous referee who gave us directions to improve this paper. Thank you Jeffrey Shallit for your great patience during the reviewing process of this paper.


\bigskip
\hrule
\bigskip

\noindent 2020 {\it Mathematics Subject Classification}: Primary 05A18.

\noindent \emph{Keywords:} noncrossing partition.

\bigskip
\hrule
\bigskip

\noindent (Concerned with sequences
\seqnum{A363448} and \seqnum{A363449}.)

\end{document}